\documentclass[11pt, a4paper]{article}

\usepackage[utf8]{inputenc}
\usepackage[T1]{fontenc}
\usepackage{amsmath, amssymb, amsthm}
\usepackage{mathrsfs}
\usepackage{geometry}
\usepackage{tikz-cd} 
\usepackage{enumitem}
\usepackage{hyperref}
\usepackage{booktabs}
\usepackage{array}

\newif\ifreviewversion
\reviewversionfalse
\ifreviewversion
  \usepackage[mathlines]{lineno}
\fi

\hypersetup{
  pdftitle={On Automorphism Groups of (1,2)-Surfaces},
  pdfauthor={Hang Zhao},
  pdfsubject={Automorphism groups of surfaces of general type},
  pdfkeywords={surfaces of general type, automorphism groups, genus-two fibrations, double covers, Hirzebruch surfaces}
}

\newtheorem{theorem}{Theorem}[section]
\newtheorem{lemma}[theorem]{Lemma}
\newtheorem{proposition}[theorem]{Proposition}
\newtheorem{corollary}[theorem]{Corollary}

\newtheorem{remark}[theorem]{Remark}
\newtheorem*{theoremA}{Theorem A}
\newtheorem*{theoremB}{Theorem B}

\title{On Automorphism Groups of $(1,2)$-Surfaces}
\author{Hang Zhao}
\date{}

\begin{document}
	\ifreviewversion
	  \linenumbers
	\fi
	\maketitle
	
	\begin{abstract}
		Let $S$ be a minimal surface of general type with $K_S^2=1$ and $p_g(S)=2$, called a $(1,2)$-surface. We give a geometric recovery and refinement of D.~Wen's bound $|\mathrm{Aut}(S)|\leq 200$ using the canonical genus-two fibration and the double-cover model over the Hirzebruch surface $\Sigma_2$. Computing $\mathrm{Aut}(\Sigma_2)$ from its Cox ring separates the action on the base $\mathbb{P}^1$ from the vertical kernel. If $\overline K_0$ denotes the vertical image on $\Sigma_2$, then $\overline K_0$ is cyclic of order $5$, $4$, $2$, or $1$, and we prove the respective sharp bounds
		\[
		|\mathrm{Aut}(S)|\leq 200,\quad 192,\quad 48,\quad 80.
		\]
		The proof combines singularity indices of genus-two fibrations with invariant divisors arising from the branch equation. We also construct admissible branch divisors attaining all four bounds and show that their full automorphism groups are, respectively,
$$
\begin{gathered}
(C_{10}\times C_{10})\rtimes C_2,\qquad
(BO_{48}\times C_8)/\langle(-I,J^4)\rangle,\\
(BT_{24}\times C_4)/\langle(-I,J^2)\rangle,\qquad
C_{40}\times C_2.
\end{gathered}
$$
Here $BO_{48}$ and $BT_{24}$ denote the binary octahedral and binary tetrahedral groups, respectively, and $J$ denotes a generator of the corresponding cyclic factor.
	\end{abstract}


	\section{Introduction}
	
	The study of automorphism groups of surfaces of general type is a classical problem in algebraic geometry. For a minimal surface $S$ of general type, the automorphism group $\mathrm{Aut}(S)$ is finite. A specific class of interest is the family of surfaces with $K_S^2=1$ and $p_g(S)=2$, denoted as $(1,2)$-surfaces.

	Beyond their intrinsic appeal as surfaces lying on the second Noether line $K^2=2p_g-3$ with the smallest possible invariants, $(1,2)$-surfaces also arise naturally in the explicit birational geometry of threefolds of general type. The Noether inequality for projective threefolds of general type, $\mathrm{vol}(X)\geq\tfrac{4}{3}p_g(X)-\tfrac{10}{3}$, was established through a sequence of works. J.~A.~Chen, M.~Chen and C.~Jiang first proved it for $p_g(X)\leq 4$ and $p_g(X)\geq 21$ \cite{ChenChenJiang2020Noether}, and their addendum improved the range to $p_g(X)\geq 11$ \cite{ChenChenJiang2020Addendum}. The remaining cases $5\leq p_g(X)\leq 10$ were settled by M.~Chen, Y.~Hu and C.~Jiang \cite{ChenHuJiang2025NoetherModuli}, completing the proof of the inequality in full generality. The inequality is sharp: M.~Kobayashi \cite{Kobayashi1992Noether} constructed examples attaining equality, lying on the three-dimensional \emph{Noether line} $\mathrm{vol}(X)=\tfrac{4}{3}p_g(X)-\tfrac{10}{3}$. Such extremal threefolds are closely tied to $(1,2)$-surfaces. Y.~Hu and T.~Zhang proved that the sharp lower bound for the slope of fibred threefolds over curves is $\tfrac{4}{3}$, with equality occurring precisely when the general fibre is a $(1,2)$-surface \cite{HuZhang2021FiberedLowSlope}, and they further classified the equality case of the Noether inequality for $p_g(X)\geq 11$ by describing the explicit structure of the relative canonical model \cite{HuZhang2025SmallVolume}. S.~Coughlan and R.~Pignatelli introduced and studied \emph{simple fibrations in $(1,2)$-surfaces}, showing that almost all Gorenstein simple fibrations over $\mathbb{P}^1$ with at worst canonical singularities give canonical threefolds on the Noether line \cite{CoughlanPignatelli2023SimpleFibrations}. More recently, S.~Coughlan, Y.~Hu, R.~Pignatelli and T.~Zhang proved, in the setting of the refined Noether line, that such threefolds admit simple fibrations in $(1,2)$-surfaces up to crepant birational modification, and used this description to study the corresponding moduli spaces \cite{CoughlanHuPignatelliZhang2024Moduli}. Thus the moduli, canonical models and automorphism theory of $(1,2)$-surfaces provide essential input for understanding threefolds near the Noether line.

	Recently, D.~Wen \cite{Dav21OneTwo} classified the automorphism groups of these surfaces.
	
	\begin{theorem}[D.~Wen \cite{Dav21OneTwo}]
		Let $S$ be a minimal surface of general type with $K_S^2=1$ and $p_g(S)=2$. Then:
		\begin{enumerate}
			\item $|\mathrm{Aut}(S)| \leq 200$.
			\item $\mathrm{Aut}(S)$ is isomorphic to groups of the form $C_m$, $C_m \times C_n$, or $C_m \rtimes T$, with specific constraints on $m, n$ and $T$.
		\end{enumerate}
	\end{theorem}
	
	Wen's approach relies on the analysis of the canonical ring $R(S)=\bigoplus_{m\geq 0}H^0(S,\mathcal{O}_S(mK_S))$ and the representation of $\mathrm{Aut}(S)$ as a subgroup of $\mathrm{GL}_2(\mathbb{C})$ acting on the canonical model in weighted projective space $\mathbb{P}(1,1,2,5)$.
	
	In this paper we give a geometric proof of Wen's bound and show that it is sharp. The canonical pencil $|K_S|$ induces a genus-two fibration on the blow-up of $S$, yielding an exact sequence $1\to K_0\to \mathrm{Aut}(S)\to H_0\to 1$ in which $H_0\subset \mathrm{PGL}_2(\mathbb{C})$ acts on the base $\mathbb{P}^1$ and $K_0$ acts fiberwise; we write $\overline{K}_0$ for the image of $K_0$ in $\mathrm{Aut}(\Sigma_2)$. Our first result recovers the bound in a form refined according to $\overline{K}_0$.

\begin{theoremA}
	Let $S$ be a $(1,2)$-surface. Then $|\mathrm{Aut}(S)|\leq 200$.
\end{theoremA}

More precisely, we establish the sharper bound (Theorem~\ref{thm: main})
\[
|\mathrm{Aut}(S)|\leq
\begin{cases}
200 & \text{if }\overline{K}_0\cong C_5,\\
192 & \text{if }\overline{K}_0\cong C_4,\\
48 & \text{if }\overline{K}_0\cong C_2,\\
80 & \text{if }\overline{K}_0=1.
\end{cases}
\]

Our second result shows that each of these bounds is attained; in particular, the four bounds above are sharp.

\begin{theoremB}
	Each of the four bounds above is attained as the order of the full automorphism group of a $(1,2)$-surface. The groups
	$$
	\begin{gathered}
	(C_{10}\times C_{10})\rtimes C_2,\qquad
	(BO_{48}\times C_8)/\langle(-I,J^4)\rangle,\\
	(BT_{24}\times C_4)/\langle(-I,J^2)\rangle,\qquad
	C_{40}\times C_2,
	\end{gathered}
	$$
	of orders $200$, $192$, $48$, and $80$, respectively, each arise as the full automorphism group $\mathrm{Aut}(S)$ of some $(1,2)$-surface $S$, in the respective cases $\overline{K}_0\cong C_5,\,C_4,\,C_2,\,1$. Here $BO_{48}$ and $BT_{24}$ denote the binary octahedral and binary tetrahedral groups, respectively, and $J$ denotes a generator of the corresponding cyclic factor.
\end{theoremB}

We now outline the proof. Since $K_S^2=1$ and $p_g(S)=2$, the canonical system $|K_S|$ is a pencil with a single base point. Blowing up this point resolves the pencil into a relatively minimal genus-two fibration $f\colon\widetilde S\to\mathbb P^1$. Every automorphism of $S$ preserves $|K_S|$ and fixes its base point, hence lifts to an automorphism of $\widetilde S$ commuting with $f$. This identifies $\mathrm{Aut}(S)$ with a subgroup of the fibration automorphism group $\mathrm{Aut}(f)$ and reduces the problem to bounding the latter.

Applying the exact sequence~\eqref{eq: vert horiz seq} of \S\ref{subsec: fibration aut} to $G=\mathrm{Aut}(S)$ gives
\[
1\to K_0\to \mathrm{Aut}(S)\to H_0\to 1,
\]
in which the horizontal part $H_0\subseteq\mathrm{PGL}_2(\mathbb C)$ is the image acting on the base $\mathbb P^1$ and the vertical part $K_0$ acts fiberwise. Thus $|\mathrm{Aut}(S)|=|K_0|\cdot|H_0|$, and it suffices to bound the two factors separately.

To make both factors computable, we use the bicanonical model: $\widetilde S$ is realized as the canonical resolution of a double cover of the Hirzebruch surface $\Sigma_2$ branched along
\[
R=\Delta_0+R_0,\qquad R_0\in |5\Delta_0+10\Gamma|.
\]
We compute $\mathrm{Aut}(\Sigma_2)$ from its Cox homogeneous coordinate ring, which canonically separates the automorphisms induced on the base $\mathbb P^1$ from those acting trivially on it. An automorphism of $S$ descends to an automorphism of $\Sigma_2$ preserving $R$, so that $H_0$ and the image $\overline K_0$ of $K_0$ in $\mathrm{Aut}(\Sigma_2)$ are both realized inside $\mathrm{Aut}(\Sigma_2)$.

The vertical part is bounded as follows: the deck involution of the double cover gives $|K_0|=2|\overline K_0|$, and Chen's local classification of genus two branch fibers \cite{ChenZhi1994Genus2} forces $\overline K_0$ to be cyclic of order in $\{1,2,4,5\}$. For the horizontal part, we combine Xiao's singularity indices \cite{Xia85}, which constrain the singular fibers through the relation $\sum s_2=40$, with the binary forms cut out on the base by the branch equation and the Beauville--Xiao lower bound of at least three singular fibers for a genus-two fibration over $\mathbb P^1$ on a surface of general type. Carrying this out in each of the four cases for $\overline K_0$ yields the refined bounds of Theorem~A. Finally, for Theorem~B we exhibit explicit branch divisors with prescribed symmetry whose double covers realize the four groups above as the full automorphism group.

Theorem~A is proved as Theorem~\ref{thm: main}, and Theorem~B as Proposition~\ref{prop: large branch examples} in \S\ref{sec: large branch curves}; in particular the $\overline{K}_0\cong C_5$ construction there attains $|\mathrm{Aut}(S)|=200$.
	
	\section{Preliminaries}
	
	\subsection{Automorphism group of a fibration}\label{subsec: fibration aut}
	A fibration $f\colon S\to B$ is a surjective morphism from a complex smooth projective surface $S$ to a smooth curve $B$ with connected fibers. A fibration is called relatively minimal if there are no $(-1)$-curves contained in fibers of $f$. An automorphism of the fibration $f\colon S\to B$ is a pair of automorphisms $(\sigma,\tau)\in\mathrm{Aut}(S)\times\mathrm{Aut}(B)$ such that the following diagram commutes
	\[\begin{tikzcd}
		S\ar[r,"\sigma"] \ar[d,"f"'] & S\ar[d,"f"]\\
		B \ar[r,"\tau"] & B
	\end{tikzcd}\]
	We denote by $\mathrm{Aut}(f)$ the group of automorphisms of the fibration $f$. The homomorphism
	\[
	\psi\colon\mathrm{Aut}(f)\to \mathrm{Aut}(B),\qquad(\sigma,\tau)\mapsto \tau,
	\]
	has kernel $\mathrm{Aut}_B(S)$. Therefore, every finite subgroup $G$ of $\mathrm{Aut}(f)$ fits into an exact sequence
	\begin{equation}\label{eq: vert horiz seq}
		1\to K\to G\xrightarrow{\psi} H\to 1,
	\end{equation}
We call $K$ (resp. $H$) the vertical (resp. horizontal) part of $G$. We have $|G|=|K|\cdot |H|$. The group $K$ naturally acts on fibers of $f$.

We mainly consider $B=\mathbb{P}^1$. Hence $H$ is a finite subgroup of $\mathrm{Aut}(\mathbb{P}^1)=\mathrm{PGL}_2(\mathbb{C})$ and is one of the following groups.
\begin{table}[htbp]
	\centering
	\caption{Finite subgroups of $\mathrm{PGL}_2(\mathbb{C})$ and possible orbit lengths.}
	\begin{tabular}{@{} l c c @{}}
		\toprule
		$H$ & $|H|$ & \textbf{Number of points in an orbit} \\
		\midrule
		Cyclic group $C_m$ & $m$ & $1$, $m$ \\
		Dihedral group $D_{2m}$ & $2m$ & $2$, $m$, $2m$ \\
		Tetrahedral group $T_{12}$ & $12$ & $4$, $6$, $12$ \\
		Octahedral group $O_{24}$ & $24$ & $6$, $8$, $12$, $24$ \\
		Icosahedral group $I_{60}$ & $60$ & $12$, $20$, $30$, $60$ \\
		\bottomrule
	\end{tabular}
\end{table}

Note that $T_{12}\cong A_4, O_{24}\cong S_4$ and $I_{60}\cong A_5$.

Let $n$ be the number of singular fibers of $f$. We have the following result.

\begin{lemma}\label{lem: bound of H for C_m D_2m}
	If $H\cong C_m,D_{2m}$, then $|H|\leq 2n$. Here we assume $n\geq 3$.
\end{lemma}

\begin{proof}
Let $\Delta$ be the locus of singular fibers on $B$. Then $\Delta$ is preserved by $H$, hence it is a union of $H$-orbits.

If $H=C_m$, it has two fixed points on $B$. Since $f$ has at least $3$ singular fibers, we have 
\[\#\Delta=n=km+r,\quad r=0,1,2, k\geq1.\]
Therefore $|H|=m\leq km+r=n$.

If $H=D_{2m}$, it has one orbit of two points with stabilizer $C_m$ and two orbits of $m$ points with stabilizer $C_2$. Hence
\[\#\Delta =n= k(2m)+2r+sm, r=0,1,s=0,1,2.\]
If $k\geq 1$, then $|H|=2m\leq 2km+2r+sm=n$.

If $k=0$, then $n=2r+sm, s\geq 1$. In this case $|H|=2m\leq 2sm+4r=2n$.
\end{proof}

\subsection{Genus-two fibrations and singularity indices}

Let $f\colon S\to B$ be a relatively minimal genus-two fibration, i.e., the general fiber has genus two. Then there is a ruled surface $\tilde{\varphi}\colon \widetilde{P}\to B$ (not necessarily relatively minimal) and a double cover $\tilde{\theta}\colon \widetilde{S}\to \widetilde{P}$ branched over a smooth curve $\widetilde{R}\in |2\tilde{\mathfrak{d}}|$ with $\mathfrak{d}\in\mathrm{Pic}(\widetilde{P})$ such that $f$ is the relatively minimal model of $\tilde{f}=\tilde{\varphi}\circ\tilde{\theta}$. We have the following commutative diagram
\[\begin{tikzcd}
	S  \ar[drr,"f"'] & \widetilde{S} \ar[l,"\rho"']\ar[r,"\tilde{\theta}"] \ar[dr,"\tilde{f}"] & \widetilde{P} \ar[d,"\tilde{\varphi}"] \\
	& & B
\end{tikzcd}\]
where $\rho$ contracts the $\tilde{f}$-vertical $(-1)$-curves on $\widetilde{S}$.

Let $\varphi\colon P\to B$ be the relatively minimal model of $\tilde{\varphi}$. Then $P\cong \mathbb{P}(f_*\omega_{S/B}\otimes\mathscr{L})$ for some ample invertible sheaf $\mathscr{L}$ on $B$. Let $R=\tilde{\psi}_*\widetilde{R}$. Then there is a divisor $\mathfrak{d}$ on $P$ such that $R\sim 2\mathfrak{d}$. We call the triple $(P,R,\mathfrak{d})$ the genus-two data for $f$. This data uniquely determines $f$.


There is a natural contraction $\tilde{\psi}\colon \widetilde{P}\to P$, which is a composition of blowing-ups, and we can write 
\[\tilde{\psi}\colon P_N\xrightarrow{\psi_N} P_{N-1}\xrightarrow{\psi_{N-1}}\cdots\xrightarrow{\psi_2}P_1\xrightarrow{\psi_1} P_0,\]
where $P_N=\widetilde{P}$, $P_0=P$, and $\psi_i\colon P_i\to P_{i-1}$ is the blow-up at $x_i\in P_{i-1}$. We define a reduced divisor $R_i$ inductively as $R_{i-1}=\psi_{i*}R_i$ starting from $R_N=\widetilde{R}$ down to $R_0=R$.
We set $E_i=\psi_i^{-1}(x_i)$ and $m_i=\mathrm{mult}_{x_i}R_{i-1}$, then $R_i=\psi_i^*R_{i-1}-2[\frac{m_i}{2}]E_i$ and there is a divisor $\mathfrak{d}_i$ on each $P_i$ such that $R_i\sim 2\mathfrak{d}_i$, hence $\mathfrak{d}_i=\psi_i^*\mathfrak{d}_{i-1}-[\frac{m_i}{2}]E_i$.

Let $\theta_i\colon S_i\to P_i$ be the double cover associated to the data $(R_i,\mathfrak{d}_i)$ for each $i=0,\dots,N$. We choose $\tilde{\psi}$ to be the minimal resolution of the singular points of $R$; then $\widetilde{S}\to S_0$ is the minimal resolution of $S_0$. A singular point of $R$ is called negligible if the point itself and all its infinitely near points are double points or triple points with at least two distinct tangents \cite[Section~1]{ChenZhi1994Genus2}. Let $\hat{\psi}\colon \widehat{P}\to P$ be the composition of blowing-up of non-negligible singular points of $R$, and $(\widehat{R},\hat{\mathfrak{d}})$ be the corresponding double covering data on $\widehat{P}$. The principal part $\widehat{R}^0$ is defined to be the divisor $\widehat{R}$ minus all isolated vertical $(-2)$-curves from it.

Let $F,\widetilde{F},\widetilde{\Gamma}$ and $\Gamma$ be general fibers of $f,\tilde{f},\tilde{\varphi}$ and $\varphi$ respectively. For fibers over a specified point $p\in B$, we write $F_p$, $\widetilde F_p$, $\widetilde\Gamma_p$, and $\Gamma_p$. For any fixed $p\in B$, we consider all singular points (including infinitely near ones) of $R$ on $\Gamma_p$. The second and third singularity indices $s_2(F_p)$ and $s_3(F_p)$ are defined as follows:

If $R$ has no quadruple singularities on $\Gamma_p$, then $s_3(F_p)$ equals the number of $(3\to 3)$ type singularities of $R$ on $F$. Otherwise $s_3(F_p)$ equals the number of $(3\to 3)$ type singularities on $F$ plus one. Hence $s_3(F_p)=0$ if and only if $R$ has no non-negligible singularities on $\Gamma_p$.

Let $\hat{\varphi}\colon \widehat{R}^0\to B$ be the natural projection induced by $\hat{\varphi}=\varphi\circ\hat{\psi}\colon \widehat{P}\to B$. Then the second singularity index $s_2(F_p)$ of $F_p$ is the ramification index of the divisor $\widehat{R}^0$ on $p$ with respect to the projection $\hat{\varphi}$. We define the singularity index of the genus-two fibration $f$ to be
\[s_k(f)=\sum_{p\in\Delta}s_k(F_p), k=2,3,\]
where $\Delta$ is the locus of singular fibers of $f$ on $B$. These invariants are related to the relative invariants of $f$ as follows.

\begin{theorem}\label{thm: Xiao's formula}\cite{Xia85}
	Let $f\colon S\to B$ be a relatively minimal genus-two fibration. Then we have
	\[K_{S/B}^2=\frac{1}{5}s_2(f)+\frac{7}{5}s_3(f).\]
\end{theorem}

	\subsection{The geometry of \texorpdfstring{$(1,2)$}{(1,2)}-surfaces}
	\label{subsec: geometry}
	
	Let $S$ be a $(1,2)$-surface. The canonical linear system $|K_S|$ is a pencil with a unique base point $b \in S$. Let $\epsilon: \widetilde{S} \to S$ be the blow-up of $S$ at $b$. The induced map $f = \Phi_{K_{\widetilde{S}}}: \widetilde{S} \to \mathbb{P}^1$ is a genus-two fibration.
	
	The double-cover description we use is the standard one for surfaces on the Noether line, a special case of Horikawa's classification of algebraic surfaces of general type with small $c_1^2$ \cite{Horikawa1976SmallI,Horikawa1976SmallII}. More precisely, Horikawa's construction imposes the condition that the branch curve have no infinitely near triple points \cite[Lemma~2.1 and Theorem~2.1, pp.~128--129]{Horikawa1976SmallII}. In the standard terminology of double covers, double or triple branch points with no infinitely near triple point are called non-essential, or simple \cite[p.~481]{CilibertoMendesLopesPardini2014Classification}; they belong to Chen's negligible class, in which the point and all its infinitely near points are double points or triple points with at least two distinct tangents \cite[Section~1]{ChenZhi1994Genus2}. Thus the modern term \emph{negligible} used below records the singularity class relevant to Chen's singularity indices, rather than quoting Horikawa's terminology verbatim.

	Furthermore, the bicanonical map of $\widetilde{S}$ induces a generically double cover $\tau: \widetilde{S} \to \Sigma_2$ onto the Hirzebruch surface $\Sigma_2$. The branch locus $R$ is a disjoint union of $\Delta_0$ and $R_0$ with $R_0\in|5\Delta_0+10\Gamma|$, where $\Delta_0$ is the section with self-intersection $-2$ and $\Gamma$ is the fiber class, and $R$ has only negligible singularities. Then $\widetilde{S}$ is the minimal resolution of the double cover $S'\to \Sigma_2$ associated to the data $(R,3\Delta_0+5\Gamma)$. We have $\tau^*\Delta_0=2E$, where $E$ is the exceptional curve of $\epsilon$, and $\widetilde{S}$ has invariants $p_g(\widetilde{S})=2$ and $K_{\widetilde{S}}^2=0$.
	
	We have the following commutative diagram:
	\[
		\begin{tikzcd}
		& \widetilde{S} \ar[r,"\widetilde{\Phi}"] \ar[d,"\epsilon"] \ar[dl,"f"']& \Sigma_2\ar[d,"\rho"]\\
		\mathbb{P}^1 & S \ar[r, dashed, "\Phi_2"] & \Sigma'_2
	\end{tikzcd}
	\]
	where $\Phi_2$ is the bicanonical map of $S$, and $\Sigma_2'$ is the rational normal cone over a rational curve of degree $2$.
The fibration $f\colon \widetilde{S}\to\mathbb{P}^1$ is induced by the natural projection $\pi\colon \Sigma_2\to\mathbb{P}^1$, and $E$ is a section of $f$. Since $f$ is the canonical map of $\widetilde{S}$ and $F\sim K_{\widetilde{S}}-E$ for a general fiber $F$, adjunction gives $2g(F)-2=(K_{\widetilde{S}}+F)\cdot F=2$.
Since $R$ has only negligible singularities, there are no vertical $(-1)$-curves on $\widetilde{S}$; therefore $f\colon\widetilde{S}\to\mathbb{P}^1$ is a relatively minimal genus-two fibration.

Conversely, this construction can be reversed, so that $(1,2)$-surfaces arise precisely in this way. The linear system $|\Delta_0+2\Gamma|$ is base-point-free by \cite[V, Theorem~2.17]{Har77}, hence so is its fifth multiple $|5\Delta_0+10\Gamma|=|5(\Delta_0+2\Gamma)|$. By Bertini's theorem a general $R_0\in|5\Delta_0+10\Gamma|$ is smooth and irreducible, and since $R_0\cdot\Delta_0=0$ it is disjoint from $\Delta_0$. For any reduced $R=\Delta_0+R_0$ with only negligible singularities, the double cover of $\Sigma_2$ associated to the data $(R,3\Delta_0+5\Gamma)$ has minimal resolution $\widetilde{S}$ with $p_g(\widetilde{S})=2$ and $K_{\widetilde{S}}^2=0$; the inverse image of $\Delta_0$ is a $(-1)$-curve, and contracting it yields a minimal surface $S$ of general type with $K_S^2=1$ and $p_g(S)=2$, that is, a $(1,2)$-surface.

We will study the automorphisms of $S$ via the above geometric construction.

\begin{remark}\label{rem: n geq 3}
	The fibration $f\colon \widetilde{S}\to\mathbb{P}^1$ has at least $3$ singular fibers. First, the residual horizontal branch $R_0\to\mathbb{P}^1$ is not \'{e}tale. Indeed, an \'{e}tale cover of $\mathbb{P}^1$ is trivial, so otherwise $R_0$ would split into five disjoint sections. Each such section is disjoint from $\Delta_0$ and hence has numerical class $\Delta_0+2\Gamma$, but two sections of this class have intersection number $(\Delta_0+2\Gamma)^2=2$, a contradiction. Chen's analysis then identifies $f$ as a genus-two fibration of variable moduli \cite[proof of Proposition~3.4]{ChenZhi1994Genus2}. Beauville proves, without any semistability hypothesis, that a variable-moduli family of curves over $\mathbb{P}^1$ has at least three singular fibers \cite[Proposition~1(1), p.~97]{Bea81}. The assertion follows.
\end{remark}

	\begin{lemma}\label{lem: lift to S}
		Any automorphism $\sigma \in \mathrm{Aut}(S)$ lifts to an automorphism of $\widetilde{S}$ preserving the fibration $f$ (an automorphism of $f$) and the section $E$.
	\end{lemma}
	
	\begin{proof}
		The base point $b$ is unique, hence fixed by any $\sigma \in \mathrm{Aut}(S)$. Thus $\sigma$ lifts to $\widetilde{S}$. Since $|K_S|$ is the canonical system, it is invariant under automorphisms, implying $\sigma$ preserves the fibration structure.
	\end{proof}
	
	Consequently, $\mathrm{Aut}(S)$ can be identified with the subgroup of $\mathrm{Aut}(f)$ that stabilizes the section $E$. We consider the natural exact sequence:
	\begin{equation}
		1 \to K_0 \to \mathrm{Aut}(S) \to H_0 \to 1
	\end{equation}
	where $H_0 \subset \mathrm{Aut}(\mathbb{P}^1)$ is the horizontal part acting on the base, and $K_0$ is the vertical part acting fiberwise.

\subsection{Automorphisms of the Hirzebruch surface \texorpdfstring{$\Sigma_2$}{Sigma2}}
\label{subsec: aut sigma2}

We recall the automorphism group of $\Sigma_2$ in the form needed below. The computation follows Cox's homogeneous-coordinate construction for toric varieties \cite{Cox1995HomogeneousCoordinateRing}, in the corrected form of his automorphism-group statement \cite{Cox2014Erratum}: for a complete simplicial toric variety the identity component of the automorphism group is recovered from the graded automorphisms of the Cox ring modulo the characteristic quasitorus. For background on toric varieties and homogeneous coordinates we refer to \cite{CoxLittleSchenck2011ToricVarieties}, and to \cite{ArzhantsevDerenthalHausenLaface2015CoxRings} for Cox rings in general. View $\Sigma_2$ as the toric surface with fan generated by
\[
n_1=(1,0),\quad n_2=(0,1),\quad n_3=(-1,2),\quad n_4=(0,-1).
\]
Its Cox ring is
\[
S=\mathbb{C}[x_1,x_2,x_3,x_4],
\]
graded by $\mathrm{Cl}(\Sigma_2)\cong\mathbb{Z}^2$ with
\[
\deg x_1=(1,0),\quad \deg x_2=(0,1),\quad
\deg x_3=(1,0),\quad \deg x_4=(2,1).
\]
Here $x_1,x_3$ are homogeneous coordinates on the base $\mathbb{P}^1$, while $x_2,x_4$ are fiber coordinates. The divisor $\Delta_0$ is $D_2=\{x_2=0\}$, and the fiber class is $\Gamma=D_1\sim D_3$.

The degree pieces relevant to graded automorphisms are
\[
S_{(1,0)}=\langle x_1,x_3\rangle,\qquad
S_{(0,1)}=\langle x_2\rangle,\qquad
S_{(2,1)}=\langle x_4,x_1^2x_2,x_1x_3x_2,x_3^2x_2\rangle .
\]
Equivalently, the Demazure roots split into the semi-simple roots
\[
R_s=\{(1,0),(-1,0)\},
\]
those whose negatives are again roots, and the unipotent roots
\[
R_u=\{(0,-1),(-1,-1),(-2,-1)\},
\]
those whose negatives are not. The semi-simple roots account for the reductive part of $\mathrm{Aut}(\Sigma_2)$ (the factor $\mathrm{PGL}_2(\mathbb{C})\times\mathbb{C}^*$ below), while the unipotent roots account for its unipotent radical; together with the characteristic quasitorus they determine the structure of $\mathrm{Aut}(\Sigma_2)$ given below. Thus $\dim \mathrm{Aut}_g(S)=9$. After quotienting by the characteristic quasitorus
\[
G=\mathrm{Hom}(\mathrm{Cl}(\Sigma_2),\mathbb{C}^*)\cong(\mathbb{C}^*)^2,
\]
one obtains
\[
\dim \mathrm{Aut}(\Sigma_2)=7.
\]
More explicitly, by \cite[\S4]{Cox1995HomogeneousCoordinateRing},
\[
\mathrm{Aut}(\Sigma_2)\cong
H^0(\mathbb{P}^1,\mathcal{O}_{\mathbb{P}^1}(2))\rtimes\bigl(\mathrm{PGL}_2(\mathbb{C})\times \mathbb{C}^*\bigr)
\cong
\mathbb{C}^3\rtimes\bigl(\mathrm{PGL}_2(\mathbb{C})\times \mathbb{C}^*\bigr).
\]
The $\mathrm{PGL}_2(\mathbb{C})$ factor acts on the base pair $(x_1,x_3)$ and the $\mathbb{C}^*$ factor scales the fiber coordinate $x_4$. The unipotent radical acts by
\[
x_4\longmapsto x_4+q(x_1,x_3)x_2,\qquad q\in H^0(\mathbb{P}^1,\mathcal{O}(2)),
\]
and is trivial on the base.

Consequently the natural projection $\pi:\Sigma_2\to\mathbb{P}^1$ induces a homomorphism
\[
\Phi:\mathrm{Aut}(\Sigma_2)\longrightarrow \mathrm{PGL}_2(\mathbb{C}).
\]
For the full Hirzebruch surface this map is surjective, with kernel
the subgroup acting trivially on the base. Thus
\begin{equation}\label{eq: aut sigma2 exact}
1\longrightarrow K\longrightarrow \mathrm{Aut}(\Sigma_2)
\xrightarrow{\Phi} \mathrm{PGL}_2(\mathbb{C})\longrightarrow 1,
\end{equation}
where
\[
K\cong \mathbb{C}^3\rtimes\mathbb{C}^*.
\]
If $\overline{G}=\mathrm{Aut}(\Sigma_2,R)$ for a branch divisor
$R=\Delta_0+R_0$, the image $\Phi(\overline{G})\subset \mathrm{PGL}_2(\mathbb{C})$ may be a proper subgroup. We write $\overline{K}_0=\overline{G}\cap K$ for the subgroup preserving $R$ and acting trivially on the base.

To bound the horizontal part $H_0$, we will first normalize a finite group of automorphisms by conjugating it into the reductive factor $\mathrm{PGL}_2(\mathbb{C})\times\mathbb{C}^*$, removing its unipotent part.

\begin{lemma}\label{lem: conjugate off unipotent}
	Let $V=H^0(\mathbb{P}^1,\mathcal{O}_{\mathbb{P}^1}(2))$ be the unipotent radical of $\mathrm{Aut}(\Sigma_2)\cong V\rtimes(\mathrm{PGL}_2(\mathbb{C})\times\mathbb{C}^*)$, acting by $x_4\mapsto x_4+q(x_1,x_3)x_2$. Then any finite subgroup $G\subset\mathrm{Aut}(\Sigma_2)$ can be conjugated by an element $v\in V$ into $\mathrm{PGL}_2(\mathbb{C})\times\mathbb{C}^*$, so that, on suitable Cox-coordinate representatives, every $g\in G$ acts by
	\[
	[x_1:x_3]\mapsto [ax_1+bx_3:cx_1+dx_3],\qquad x_4\mapsto\lambda(g)\,x_4,
	\]
	with no unipotent term.
\end{lemma}

\begin{proof}
	Write $L=\mathrm{PGL}_2(\mathbb{C})\times\mathbb{C}^*$ and let $\rho\colon G\to L$ be the projection, so each $g\in G$ has the form $(u_g,\rho(g))$ with $u_g\in V$. The multiplication rule in $V\rtimes L$ gives the cocycle identity $u_{gg'}=u_g+\rho(g)\,u_{g'}$, where $L$ acts on the complex vector space $V$ through $\rho$. Set $w=\tfrac{1}{|G|}\sum_{g\in G}u_g\in V$. Averaging the cocycle identity over $g'\in G$ yields $\rho(g)\,w=w-u_g$, i.e. $u_g=w-\rho(g)\,w$. Conjugating $G$ by the element $-w\in V$ then replaces each $u_g$ by $u_g-\bigl(w-\rho(g)\,w\bigr)=0$. Hence the conjugated subgroup lies in $L$.
\end{proof}

\begin{remark}\label{rem: conjugate branch pair}
	When $G=\mathrm{Aut}(\Sigma_2,R)$ preserves a branch divisor $R$, the conjugate $vGv^{-1}$ of Lemma~\ref{lem: conjugate off unipotent} preserves the transformed divisor $v(R)$ rather than $R$ itself. Replacing the pair $(\Sigma_2,R)$ by the isomorphic pair $(\Sigma_2,v(R))$, we may thus assume $G\subset\mathrm{PGL}_2(\mathbb{C})\times\mathbb{C}^*$ throughout, which leaves the bound on $\mathrm{Aut}(S)$ unchanged.
\end{remark}

\subsection{Equation of the branch divisor}
\label{subsec: branch equation}

With the above Cox coordinates on $\Sigma_2$, a monomial
\[
x_1^a x_2^b x_3^c x_4^d
\]
has bidegree $(a+c+2d,b+d)$. Hence $R_0\in |5\Delta_0+10\Gamma|$ is given by a homogeneous polynomial of bidegree $(10,5)$:
\[
f_0=\sum_{d=0}^{5} x_2^{5-d}x_4^d P_{10-2d}(x_1,x_3),
\]
where $P_j$ is a binary form of degree $j$. Thus the full branch divisor has equation
\begin{equation}\label{eq: branch general}
R:\quad x_2 f_0=0.
\end{equation}
The vertical scaling $x_4\mapsto \lambda x_4$ preserves a displayed branch equation precisely when all surviving $x_4$-powers acquire the same scalar. This normal form underlies the case analysis in the proof of Theorem~\ref{thm: main}.
	
	\section{Bounding the Vertical Automorphisms}
	
	In this section, we bound the order of the vertical part $K_0$. We regard an element $\alpha \in K_0$ as an element of $\mathrm{Aut}_{\mathbb{P}^1}(\widetilde{S})$ which preserves $E$. We can see that $\alpha$ restricts to an automorphism of the general fiber $F$ (a genus 2 curve) fixing the intersection point $F \cap E$. Let $\tilde{\sigma}$ be the double cover involution of $\widetilde{S}\to\Sigma_2$, since $E$ is fixed by $\tilde{\sigma}$, $\tilde{\sigma}$ descends to an automorphism of $S$, and $\tilde{\sigma}\in K_0$.

	\begin{lemma}\label{lem: K0 deck}
		Every $\alpha \in K_0$ induces a vertical automorphism of $\Sigma_2$. Writing $\overline{K}_0:=\operatorname{im}\bigl(K_0\to\mathrm{Aut}(\Sigma_2)\bigr)$ for the image, there is an exact sequence
		\[
		1\longrightarrow\langle\tilde{\sigma}\rangle\longrightarrow K_0\longrightarrow\overline{K}_0\longrightarrow 1,
		\]
		where $\tilde{\sigma}$ is the double-cover involution; in particular $|K_0| = 2|\overline{K}_0|$.
	\end{lemma}
	
	\begin{proof}
		Let $\theta'\colon S'\to \Sigma_2$ be the double cover associated with the branch locus data $(R, 3\Delta_0+5\Gamma)$. Since $S'$ has only canonical singularities, the minimal resolution $\epsilon'\colon \widetilde{S}\to S'$ contracts only vertical $(-2)$-curves. Any vertical automorphism $\alpha \in K_0$ permutes these $(-2)$-curves; thus, it descends to an automorphism of $S'$, which we still denote by $\alpha$.
		
		To show that $\alpha$ descends further to $\Sigma_2$, it suffices to prove that $\alpha$ commutes with the covering involution $\tilde{\sigma}$ of $\theta'$ (i.e., $\tilde{\sigma}\circ\alpha=\alpha\circ\tilde{\sigma}$). The restriction of $\tilde{\sigma}$ to a general fiber $F$ is the hyperelliptic involution, which is central in $\mathrm{Aut}(F)$; since $\alpha$ preserves $F$ and restricts to an automorphism of it, $\alpha|_F$ commutes with $\tilde{\sigma}|_F$. Hence $\tilde{\sigma}\circ\alpha$ and $\alpha\circ\tilde{\sigma}$ agree on the dense open union of general fibers, and being regular morphisms of $\widetilde{S}$ they coincide everywhere. Therefore $\alpha$ commutes with $\tilde{\sigma}$ and descends to the quotient $\Sigma_2$, where it gives a vertical automorphism preserving the branch divisor $R$. The kernel of the resulting homomorphism $K_0\to\overline{K}_0$ is generated by the covering involution $\tilde{\sigma}\in K_0$, which acts trivially on $\Sigma_2$; hence $|K_0|=2|\overline{K}_0|$.
	\end{proof}

	\begin{proposition}\label{prop: K0 bar cyclic}
		The group $\overline{K}_0$ is a cyclic group of order $k$, where $k \in \{1, 2, 4, 5\}$.
	\end{proposition}
	
	\begin{proof}
		Let $p\in\mathbb{P}^1$ be a general point. Since $\overline{K}_0$ acts trivially on the base, it acts on the fiber $\Gamma_p\cong\mathbb{P}^1$, preserving the branch set $R\cap\Gamma_p$ (which generically consists of $6$ distinct points) and fixing the intersection point $\Delta_0\cap\Gamma_p$ (corresponding to $E$). A finite subgroup of $\mathrm{PGL}_2(\mathbb{C})$ with a fixed point is cyclic, so $\overline{K}_0\cong C_m$ for some $m\ge 1$.

		A non-trivial element of $\mathrm{PGL}_2(\mathbb{C})$ has exactly two fixed points, one of which is $\Delta_0\cap\Gamma_p$. The remaining $R_0\cdot\Gamma = 5$ branch points of $R_0\cap\Gamma_p$ therefore decompose into $C_m$-orbits of length $m$ together with at most one further fixed point (the second fixed point of the action). Hence the number of non-fixed points among them---which is $5$ if the second fixed point is not in $R_0$, and $4$ otherwise---must be divisible by $m$.

		Hence $m$ divides $5$ (if the second fixed point does not lie in $R_0$) or $4$ (if it does), so $\overline{K}_0\cong C_m$ with $m\in\{1,2,4,5\}$.
	\end{proof}

	\begin{corollary}
		$|K_0| \leq 10$.
	\end{corollary}

	\begin{proof}
		By Lemma~\ref{lem: K0 deck} and Proposition~\ref{prop: K0 bar cyclic}, $|K_0|=2|\overline{K}_0|=2k$ with $k\in\{1,2,4,5\}$, so $|K_0|\le 10$.
	\end{proof}
	
\section{Proof of the Main Theorem}

Throughout this section $\overline{K}_0$ denotes the image of the vertical kernel $K_0$ in $\mathrm{Aut}(\Sigma_2)$; by Lemma~\ref{lem: K0 deck} one has $|K_0|=2|\overline{K}_0|$.

The bound rests on local lower bounds for the singularity index $s_2$ at the singular fibers, which we take from Chen's local classification \cite{ChenZhi1994Genus2}, translated into the present notation. Fix $p\in\mathbb{P}^1$ over which the fiber is singular and restrict $\pi\colon\Sigma_2\to\mathbb{P}^1$ to a small disk $\Delta\ni p$. This is exactly Chen's local model: a ruled surface $P_\Delta=\pi^{-1}(\Delta)$ over $\Delta$ with central fiber $F_0=\Gamma_p$, branch curve $R_\Delta=R|_{P_\Delta}$ whose horizontal part is $R'_\Delta=R_0|_{P_\Delta}$, a fiber coordinate $x$ vanishing on the section $\Delta_0$, and a base parameter $t$ centered at $p$. Chen's local vertical group $\overline{K}_\Delta\subseteq\mathrm{Aut}(P_\Delta)$ contains the global image $\overline{K}_0$, and since $\Delta_0\cdot\Gamma=1$ and $R_0\cdot\Gamma=5$, a general fiber meets $R_\Delta$ in the six points $(\Delta_0\cap\Gamma_p)\cup(R_0\cap\Gamma_p)$. When $F_p$ is singular but $F_0\not\subseteq R_\Delta$ (i.e.\ outside case (i) below), the horizontal branch $R'_\Delta$ is necessarily non-\'etale over $\Delta$: were it \'etale, $R_\Delta$ would meet every fiber transversally in six distinct points and the double cover would have smooth genus-two fibers over all of $\Delta$, contradicting the singularity of $F_p$. This is the hypothesis of parts (ii) and (iii).

\begin{proposition}\label{prop: chen local}
	Assume $R$ has only negligible singularities, so that $s_3(F_p)=0$ for all $p$. With the notation above, the second singularity index of a singular fiber satisfies the following.
	\begin{enumerate}[label=\textup{(\roman*)}]
		\item If $\overline{K}_\Delta\cong O_{24}$, $T_{12}$ or $D_{12}$, then $F_0\subseteq R_\Delta$, $R_\Delta$ has six ordinary double points on $F_0$, and $s_2(F_p)=10$ \textup{(\cite[Lemma~2.1]{ChenZhi1994Genus2})}.
		\item If $\overline{K}_\Delta\cong C_5$ and $R'_\Delta$ is not \'etale over $\Delta$, then, up to a coordinate change, $R'_\Delta$ has equation $x(x^5-t^k)$, with $s_2(F_p)\ge 6$, or $x(t^kx^5-1)$, with $s_2(F_p)\ge 4$; here $k\in\{1,2\}$ \textup{(\cite[Lemma~2.4]{ChenZhi1994Genus2})}. In the second form the exact value is $s_2(F_p)=4k$, computed below.
		\item If $\overline{K}_\Delta\cong C_4$ and $R'_\Delta$ is not \'etale over $\Delta$, then, up to a coordinate change, $R'_\Delta$ has equation $x(x^4-t^k)$ with $k\in\{1,2\}$, and $s_2(F_p)\ge 5$ \textup{(\cite[Lemma~2.6]{ChenZhi1994Genus2})}.
		\item In every case $s_2(F_p)\ge 1$ \textup{(\cite[Lemma~2.8]{ChenZhi1994Genus2})}.
	\end{enumerate}
	In particular, if $\overline{K}_0\cong C_5$ then $\overline{K}_\Delta\cong C_5$ at every singular fiber---the only group in Chen's list admitting a subgroup of order $5$---so $s_2(F_p)\ge 4$; and if $\overline{K}_0\cong C_4$ then $\overline{K}_\Delta\in\{C_4,O_{24}\}$---the only groups in Chen's list admitting a cyclic subgroup of order $4$---so $s_2(F_p)\ge 5$.
\end{proposition}

Chen's Lemma~2.4 records only $s_2(F_p)\ge4$ for the second $C_5$ form; the exact value follows from a direct ramification computation. Setting $y=1/x$, the moving branch component is $y^5=t^k$, whose ramification over the base is defined by $y^5-t^k$ and $\partial_y(y^5-t^k)=5y^4$; hence
\[
s_2(F_p)=\dim_{\mathbb C}\frac{\mathbb{C}[[t,y]]}{(y^5-t^k,\,y^4)}=\dim_{\mathbb C}\frac{\mathbb{C}[[t,y]]}{(t^k,\,y^4)}=4k,
\]
that is, $s_2(F_p)=4$ for $k=1$ and $s_2(F_p)=8$ for $k=2$.

The case $\overline{K}_0\cong C_2$ will also require the following fact about semi-invariant binary forms for the tetrahedral group.

\begin{lemma}\label{lem: A4 semiinvariant}
	Let $A_4\subset\mathrm{PGL}_2(\mathbb{C})$ be the tetrahedral group, acting on $V=\langle x_1,x_3\rangle$. Let $P_4\in\mathrm{Sym}^4 V^\vee$ be a nonzero $A_4$-semi-invariant binary quartic with character $\chi_4$, and let $P_8\in\mathrm{Sym}^8 V^\vee$ be an $A_4$-semi-invariant binary octavic with character $\chi_8$. If $\chi_8=\chi_4^2$, then $P_8\in\langle P_4^2\rangle$.
\end{lemma}

\begin{proof}
	Let $\widetilde{A}_4\subset\mathrm{SL}_2(\mathbb{C})$ be the binary tetrahedral group covering $A_4$. Since the degrees $4$ and $8$ are even, the central element $-1\in\widetilde{A}_4$ acts trivially on $\mathrm{Sym}^4 V^\vee$ and $\mathrm{Sym}^8 V^\vee$, so both descend to representations of $A_4$. Write $\mathbf{1},\rho,\rho^2$ for the three one-dimensional characters of $A_4$ (with $\rho$ of order $3$) and $W$ for the irreducible $3$-dimensional representation.

	For $g\in\mathrm{SL}_2(\mathbb{C})$ with eigenvalues $\lambda,\lambda^{-1}$, the trace of $g$ on $\mathrm{Sym}^d V^\vee$ is $\sum_{j=0}^{d}\lambda^{d-2j}$. Evaluating on the conjugacy classes of $A_4$---the identity, the double transpositions (which lift to $\lambda=i$), and the two classes of $3$-cycles (which lift to $\lambda=\zeta_6$)---gives the characters
	\[
	\chi_{\mathrm{Sym}^4 V^\vee}=(5,1,-1,-1),\qquad
	\chi_{\mathrm{Sym}^8 V^\vee}=(9,1,0,0).
	\]
	Comparing with the character table of $A_4$ yields
	\[
	\mathrm{Sym}^4 V^\vee\cong\rho\oplus\rho^2\oplus W,\qquad
	\mathrm{Sym}^8 V^\vee\cong\mathbf{1}\oplus\rho\oplus\rho^2\oplus W^{\oplus 2}.
	\]
	Thus a nonzero $A_4$-semi-invariant quartic has character $\rho$ or $\rho^2$, and in degree $8$ each one-dimensional character occurs with multiplicity one. Now $P_4^2$ is a nonzero $A_4$-semi-invariant octavic of character $\chi_4^2\in\{\rho,\rho^2\}$, so it spans the one-dimensional $\chi_4^2$-isotypic subspace of $\mathrm{Sym}^8 V^\vee$. Any $P_8$ with $\chi_8=\chi_4^2$ lies on this line, that is, $P_8\in\langle P_4^2\rangle$.
\end{proof}

The cyclic and dihedral cases of the same situation are governed by the following lemma.

\begin{lemma}\label{lem: cyclic dihedral semiinv}
	Let $H\subset\mathrm{PGL}_2(\mathbb{C})$ be cyclic or dihedral, acting on $\mathbb{P}^1=\mathbb{P}(V)$, and let $P_4\in\mathrm{Sym}^4 V^\vee$, $P_8\in\mathrm{Sym}^8 V^\vee$ be nonzero $H$-semi-invariant binary forms with characters $\chi_4,\chi_8$ satisfying $\chi_8=\chi_4^2$. Suppose $\operatorname{div}(P_4)$ is supported only on the exceptional short orbits of $H$---the two fixed points if $H$ is cyclic, the polar $2$-orbit if $H$ is dihedral. Then either $|H|\le 8$, or $P_8$ is a scalar multiple of $P_4^2$; in the latter case $\operatorname{div}(P_8)$ and $\operatorname{div}(P_4^2-4aP_8)$ are supported on that same exceptional set, for every scalar $a$.
\end{lemma}

\begin{proof}
	Suppose first $H=C_m=\langle\rho\rangle$ with $\rho[X:Y]=[\zeta_m X:Y]$, whose fixed points are $\{X=0\}$ and $\{Y=0\}$. If $\operatorname{div}(P_4)$ is supported there, then $P_4=X^rY^{4-r}$ for some $0\le r\le 4$, with $\chi_4(\rho)=\zeta_m^r$. Assume $m>8$ (otherwise $|H|\le 8$). A monomial $X^iY^{8-i}$ has $\rho$-character $\zeta_m^i$, so an octavic of character $\chi_4^2=\zeta_m^{2r}$ involves only exponents $i\equiv 2r\pmod m$; as $0\le i\le 8<m$ and $0\le 2r\le 8$, the unique such exponent is $i=2r$. Hence $P_8=c\,X^{2r}Y^{8-2r}=c\,P_4^2$.

	Now suppose $H=D_{2m}=\langle\rho,\sigma\rangle$ with $\sigma[X:Y]=[Y:X]$, whose polar orbit is $\{X=0\}\cup\{Y=0\}$. Assume $m\ge 5$ (otherwise $|H|=2m\le 8$). Since $\operatorname{div}(P_4)$ has degree $4<m$, it contains no orbit of length $m$ or $2m$, so it is supported on the polar orbit: $P_4=X^rY^{4-r}$. Invariance of $\operatorname{div}(P_4)$ under $\sigma$ forces $r=4-r$, so $P_4=(XY)^2$ and $\chi_4(\rho)=\zeta_m^2$. As above, an octavic of character $\chi_4^2=\zeta_m^4$ involves only $i\equiv 4\pmod m$ with $0\le i\le 8<2m$, whence $i=4$; thus $P_8=c\,(XY)^4=c\,P_4^2$.

	In either case $P_8=c\,P_4^2$, so $\operatorname{div}(P_8)=2\operatorname{div}(P_4)$ and $P_4^2-4aP_8=(1-4ac)P_4^2$; both have support contained in $\operatorname{supp}\operatorname{div}(P_4)$, the exceptional set.
\end{proof}

\begin{theorem}\label{thm: main}
	Let $S$ be a $(1,2)$-surface, let $f\colon\widetilde S\to\mathbb P^1$ be the genus-two fibration obtained by resolving the canonical pencil, and let
	\[
	1\to K_0\to\mathrm{Aut}(S)\to H_0\to1
	\]
	be the induced vertical--horizontal exact sequence. Then the order of $\mathrm{Aut}(S)$ is bounded as follows:
		\begin{enumerate}[label=(\roman*)]
		\item If $\overline{K}_0 \cong C_5$, then $|H_0| \leq 20$ and $|\mathrm{Aut}(S)| \leq 200$.
		\item If $\overline{K}_0 \cong C_4$, then $|H_0| \leq 24$ and $|\mathrm{Aut}(S)| \leq 192$.
		\item If $\overline{K}_0 \cong C_2$, then $|H_0| \leq 12$ and $|\mathrm{Aut}(S)| \leq 48$.
		\item If $\overline{K}_0$ is trivial, then $|H_0| \leq 40$ and $|\mathrm{Aut}(S)| \leq 80$.
	\end{enumerate}
	In particular, $|\mathrm{Aut}(S)| \leq 200$ in every case.
\end{theorem}

\begin{proof}
	To bound $|H_0|$, we employ the theory of singularity indices for genus-two fibrations developed by G.~Xiao \cite{Xia85}. By Theorem \ref{thm: Xiao's formula}, the relative canonical invariant satisfies the formula:
	\[K_{\widetilde{S}/\mathbb{P}^1}^2 = \frac{1}{5}s_2(f) + \frac{7}{5}s_3(f),\]
	where $s_2(f)$ and $s_3(f)$ are the global singularity indices. We first compute $K_{\widetilde{S}/\mathbb{P}^1}^2$. By definition, $K_{\widetilde{S}/\mathbb{P}^1} = K_{\widetilde{S}} - f^*K_{\mathbb{P}^1} = K_{\widetilde{S}} + 2F$, where $F$ is a general fiber. Hence
	\[K_{\widetilde{S}/\mathbb{P}^1}^2 = K_{\widetilde{S}}^2 + 4(K_{\widetilde{S}}\cdot F) + 4F^2 = 0 + 4(2g(F)-2) + 0 = 8,\]
	using $K_{\widetilde{S}}^2=0$, $K_{\widetilde{S}}\cdot F = 2g(F)-2=2$ (genus two fiber), and $F^2=0$. Since $K_{\widetilde{S}/\mathbb{P}^1}^2 = 8$ and the branch locus $R$ possesses no non-negligible singularities, it follows that $s_3(f) = 0$. Consequently, we obtain:
	\begin{equation}\label{eq: s2 sum}
	\sum_{p\in\Delta} s_2(F_p)=40.
	\end{equation}
	where $\Delta \subset \mathbb{P}^1$ is the set of points over which the fibers are singular, and we set $n = \#\Delta$. The group $H_0$ acts on $\Delta$. If $H_0 \cong C_m$ or $D_{2m}$, Lemma \ref{lem: bound of H for C_m D_2m} implies $|H_0| \leq 2n$. If $H_0 \cong T_{12}$, $O_{24}$, or $I_{60}$, then $|H_0| \leq 60$. We analyze each case of $\overline{K}_0$ separately below.
	
	By Proposition~\ref{prop: chen local}, the local index $s_2(F_p)$ is bounded below in terms of $\overline{K}_\Delta$, and hence in terms of $\overline{K}_0$, since $\overline{K}_0\subseteq\overline{K}_\Delta$ at every singular fiber. We analyze the cases based on $\overline{K}_0$.
	
	\textbf{Case 1: $\overline{K}_0 \cong C_5$.}
	By Proposition~\ref{prop: chen local}, $\overline{K}_\Delta\cong C_5$ at every singular fiber and $s_2(F_p) \geq 4$, which implies $4n \leq 40$, so $n \leq 10$. Referring to Table 1, we can exclude the case $H_0 \cong I_{60}$, as the smallest orbit size for $I_{60}$ is $12$, which exceeds $n$. If $H_0$ is cyclic, dihedral, or isomorphic to $T_{12}$, then $|H_0| \leq 20$. Consequently, $|\mathrm{Aut}(S)| = |K_0| \cdot |H_0| \leq 200$.
	
	We must now consider the possibility that $H_0 \cong O_{24}$. Since $R=\Delta_0+R_0$ with $R_0$ disjoint from $\Delta_0=\{x=0\}$, the five horizontal branches of $R_0$ cannot degenerate onto the $\overline{K}_\Delta$-fixed section $\Delta_0$; hence at each singular fiber only the second $C_5$ normal form $x(t^kx^5-1)=0$ of Proposition~\ref{prop: chen local}(ii) occurs, the first form $x(x^5-t^k)=0$ being excluded as it would force $R_0$ to meet $\Delta_0$. By Proposition~\ref{prop: chen local}(ii) and the ramification computation following it, these fibers have the exact second singularity indices
	\[
	s_2(F_p) = \begin{cases}
		4 & \text{if } k=1, \\
		8 & \text{if } k=2.
	\end{cases}
	\]
	Let $a$ (resp. $b$) denote the number of singular fibers of type $k=1$ (resp. $k=2$). Then, by~\eqref{eq: s2 sum}, we have $4a + 8b = 40$. The non-negative integer solutions for $(a, b)$ are:
	\[ \{(10, 0), (8, 1), (6, 2), (4, 3), (2, 4), (0, 5)\}. \]
	Since $H_0$ preserves the fiber type, both $a$ and $b$ must be sums of lengths of orbits of $H_0$ on $\Delta$. According to Table 1, the possible orbit lengths for $O_{24}$ are $6, 8, 12,$ and $24$. However, for every pair $(a, b)$ listed above, it is impossible to express both $a$ and $b$ as sums of these orbit lengths (for instance, $b$ cannot be $1, 2, 3, 4,$ or $5$, and $a$ cannot be $10$). Thus, the case $H_0 \cong O_{24}$ is excluded.
	
	\textbf{Case 2: $\overline{K}_0 \cong C_4$.}
	By Proposition~\ref{prop: chen local}, $\overline{K}_\Delta\in\{C_4,O_{24}\}$ at every singular fiber and $s_2(F_p)\geq 5$ (the local equation being $x(x^4-t^k)$, $k\in\{1,2\}$, in the $C_4$ case by Lemma~2.6 of \cite{ChenZhi1994Genus2}). This implies $5n\leq 40$, so $n\leq 8$. Since $I_{60}$ has smallest orbit of size $12>n$, it is excluded. For $H_0\cong C_m$ or $D_{2m}$, Lemma~\ref{lem: bound of H for C_m D_2m} gives $|H_0|\leq 2n\leq 16$. For $H_0\cong O_{24}$ or $T_{12}$, $|H_0|\leq 24$. Hence $|\mathrm{Aut}(S)|=|K_0|\cdot|H_0|\leq 8\times 24=192$.
	
	\textbf{Case 3: $\overline{K}_0 \cong C_2$.}
	Let $\iota$ generate $\overline{K}_0$. Choosing Cox coordinates so that $\iota\colon x_4\mapsto -x_4$, its two fixed sections are $\Delta_0=\{x_2=0\}$ and $C_\infty=\{x_4=0\}$. As the branch divisor meets a general fiber in six points and already contains $\Delta_0$, the residual divisor $R_0$ (with $R_0\cdot\Gamma=5$) must contain $C_\infty$; and invariance under $\iota$ leaves only odd powers of $x_4$. Hence
	\[R\colon\quad x_2x_4\bigl(ax_4^4+x_2^2x_4^2P_4(x_1,x_3)+x_2^4P_8(x_1,x_3)\bigr)=0,\qquad a\neq 0,\]
	with $P_4,P_8$ binary forms of degrees $4$ and $8$. Since $\overline{K}_0\cong C_2$ \emph{exactly}, necessarily $P_4\neq 0$: otherwise the surviving terms $x_4^5$ and $x_4$ would both be multiplied by $i$ under $x_4\mapsto ix_4$, so this scaling would preserve $R$ and force $C_4\subseteq\overline{K}_0$, a contradiction.

	Because $\overline{K}_0$ is normal in $\overline{G}$ and $\mathrm{Aut}(C_2)$ is trivial, $H_0$ centralizes $\iota$; hence every lift of $h\in H_0$ preserves the fixed sections $\Delta_0,C_\infty$ and acts on the fiber coordinate by a scalar. Comparing the surviving $x_4$-powers $1,3,5$, the coefficients transform by characters $h^*P_4=\chi_4(h)P_4$ and $h^*P_8=\chi_8(h)P_8$ with $\chi_8=\chi_4^2$. In particular $H_0$ preserves the nonzero effective divisor $\operatorname{div}(P_4)$ of degree $4$ on $\mathbb{P}^1$.

	We bound $H_0$ using the nonzero $H_0$-invariant divisor $\operatorname{div}(P_4)$ of degree $4$. If $H_0\cong O_{24}$ or $I_{60}$, every orbit has length at least $6$ or $12$, contradicting $\deg\operatorname{div}(P_4)=4$; these are excluded. Suppose next that $H_0$ is cyclic or dihedral. If $\operatorname{div}(P_4)$ contains a genuine orbit (one of length $m$ or $2m$), that orbit has length at most $4$, forcing $m\le 4$, so $|H_0|\le 4$ for $H_0\cong C_m$ and $|H_0|=2m\le 8$ for $H_0\cong D_{2m}$. Otherwise $\operatorname{div}(P_4)$ is supported only on the exceptional short orbits, and Lemma~\ref{lem: cyclic dihedral semiinv} applies: either $|H_0|\le 8$, or $\operatorname{div}(P_8)$ and $\operatorname{div}(P_4^2-4aP_8)$ are supported on those same $\le 2$ points, so that the singular fibers---which lie over $\{P_8=0\}\cup\{P_4^2-4aP_8=0\}$---are concentrated over at most two base points, contradicting the Beauville--Xiao lower bound of at least three singular fibers for a genus-two fibration over $\mathbb{P}^1$ on a surface of general type (Remark~\ref{rem: n geq 3}). In every cyclic or dihedral case, therefore, $|H_0|\le 8$.

	The remaining possibility is $H_0\cong T_{12}\cong A_4$. Consequently $|H_0|\leq12$, and since $|K_0|=2|\overline{K}_0|=4$, we obtain
	$$
	|\mathrm{Aut}(S)|=|K_0|\cdot|H_0|\leq4\cdot12=48.
	$$
	\textbf{Case 4: $\overline{K}_0$ is trivial.}
	Since $\overline{K}_0$ is trivial, the projection $\overline{G}\to H_0\subset\mathrm{PGL}_2(\mathbb{C})$ is injective; by Lemma~\ref{lem: conjugate off unipotent} we may assume, after conjugating by the unipotent radical (which changes neither the double cover up to isomorphism, the horizontal image, nor the triviality of $\overline{K}_0$), that every $h\in H_0$ acts by $(x_1,x_3)\mapsto h(x_1,x_3)$, $x_4\mapsto\lambda(h)x_4$, without unipotent term. Then $h$ does not mix the powers of $x_4$ in $f_0=\sum_{d=0}^5 x_2^{5-d}x_4^d P_{10-2d}(x_1,x_3)$, so each nonzero coefficient $P_{10-2d}$ is $H_0$-semi-invariant and $\operatorname{div}(P_{10-2d})$ is $H_0$-invariant. In the affine chart $t=x_3/x_1$, $u=x_4/(x_2x_1^2)$, the moving branch is $g(t,u)=\sum_{d=0}^5 p_{10-2d}(t)u^d$, of degree $5$ in $u$ (as $P_0\neq 0$, since $R_0\cap\Delta_0=\varnothing$). Its discriminant defines a nonzero $H_0$-semi-invariant section $\Delta_g$ on $\mathbb{P}^1$ whose invariant zero divisor is precisely the singular-fiber divisor and has degree at most $40$ (since $\deg_t p_{10-2d}\le 10-2d=2(5-d)$ and the discriminant of a quintic is isobaric of weight $20$ in the grading $\deg a_d=5-d$).

	Suppose first that $H_0$ is polyhedral. Reducedness of $R_0$ forces some coefficient $P_{10-2d}$ with $d<5$ to be nonzero; its invariant zero divisor has positive degree at most $10$. Since the smallest $I_{60}$-orbit has length $12$, the icosahedral case is impossible, while $T_{12}$ and $O_{24}$ already have orders at most $24$.

	Suppose next that $H_0\cong C_m$. The action has exactly two fixed points on $\mathbb{P}^1$. Since $f$ has at least three singular fibers by Remark~\ref{rem: n geq 3}, the support of $\operatorname{div}(\Delta_g)$ contains a point outside this fixed pair. Its orbit has length $m$, and therefore
	\[
	m\le \deg\operatorname{div}(\Delta_g)\le 40.
	\]

	Finally suppose that $H_0\cong D_{2m}$. If the divisor of some positive-degree coefficient $P_{10-2d}$ contains a point outside the polar $2$-orbit, it contains an orbit of length at least $m$; hence $m\le10$ and $|H_0|=2m\le20$. Otherwise every such coefficient divisor is supported on the polar pair. The reflection exchanges the two poles, so semi-invariance forces their multiplicities to agree; since $\deg P_{10-2d}=10-2d$, one has $P_{10-2d}=c_d(x_1x_3)^{5-d}$. Consequently
	\[
	g(t,u)=\sum_d c_d t^{5-d}u^d=t^5Q(u/t),
	\]
	whose discriminant is supported only at $0$ and $\infty$. This would give at most two singular fibers, contrary to Remark~\ref{rem: n geq 3}. Thus the exceptional dihedral alternative cannot occur.

	In all cases $|H_0|\le 40$, and since $|K_0|=2|\overline{K}_0|=2$,
	\[|\mathrm{Aut}(S)|=|K_0|\cdot|H_0|\le 2\cdot 40 = 80.\]
\end{proof}

\section{Explicit Branch Divisors and Large Automorphism Groups}
\label{sec: large branch curves}

We now pass from fibration bounds to the explicit branch divisor on $\Sigma_2$. Let
\[
R=\Delta_0+R_0,\qquad R_0\in |5\Delta_0+10\Gamma|,
\]
and assume that $R$ is reduced and has only negligible singularities. Let
\[
\overline{G}=\mathrm{Aut}(\Sigma_2,R),\qquad \overline{K}_0=\overline{G}\cap K,
\]
where $K$ is the kernel in \eqref{eq: aut sigma2 exact}. The automorphism group of the minimal model $S$ is obtained by lifting $\overline{G}$ to the double cover and then contracting the unique $(-1)$-curve lying over $\Delta_0$. The branch divisors below are written in the Cox coordinates of \S\ref{subsec: branch equation}.

For the branch divisors below, set
$$
\Phi_8=x_1^8+14x_1^4x_3^4+x_3^8,\qquad
\Phi=x_1^4+2i\sqrt3\,x_1^2x_3^2+x_3^4.
$$
The group $BO_{48}$ is the binary octahedral group, $J$ generates $C_8$, and the quotient $(BO_{48}\times C_8)/\langle(-I,J^4)\rangle$ identifies $-I$ with $J^4$. In $(C_{10}\times C_{10})\rtimes C_2$, the factor $C_2$ interchanges the two $C_{10}$ factors. Writing $x=(x_1,x_3)^t$, let $G_\Phi$ be the group of triples $(M,b,d)\in\mathrm{GL}_2(\mathbb C)\times(\mathbb C^*)^2$ satisfying
$$
\Phi(Mx)=b^2\Phi(x),\qquad d^2=b^5,
$$
modulo the central equivalence
$$
(M,b,d)\sim(\lambda M,\lambda^2b,\lambda^5d),\qquad\lambda\in\mathbb C^*.
$$
The implicit nonzero scalars in the first, second and fourth rows of the following table are chosen generally.

\begin{proposition}\label{prop: large branch examples}
	Each branch divisor $R\colon x_2f_0=0$ in the table below is reduced with only negligible singularities and defines a $(1,2)$-surface with the indicated exact vertical image and full automorphism group. Each row attains the corresponding bound of Theorem~\ref{thm: main}.
	$$
\renewcommand{\arraystretch}{1.4}
	\begin{array}{c|c|c|c}
	\overline{K}_0 & f_0 & \mathrm{Aut}(S) & |\mathrm{Aut}(S)|\\
	\hline
	C_5 & x_2^5(x_1^{10}+x_3^{10})+x_4^5 & (C_{10}\times C_{10})\rtimes C_2 & 200\\
	C_4 & x_2^4x_4\,\Phi_8+x_4^5 & (BO_{48}\times C_8)/\langle(-I,J^4)\rangle & 192\\
	C_2 & x_4\bigl(x_4^4+x_2^2x_4^2\Phi+x_2^4\Phi^2\bigr) & G_\Phi & 48\\
	1 & x_4^5+x_1^8x_2^4x_4+x_2^5x_3^{10} & C_{40}\times C_2 & 80
	\end{array}
	$$
\end{proposition}

\begin{proof}
	We first justify that the vertical images in the table are computed by scalings, rather than by general affine transformations of a fiber. On the chart used below, every displayed moving branch equation is a monic quintic $g(t,u)$ whose $u^4$-coefficient is zero. Thus, on every smooth fiber, its five roots $r_1(t),\ldots,r_5(t)$ have sum zero. A vertical automorphism acts as $u\mapsto\lambda u+q(t)$; if it preserves the unordered root set, then
	\[
	0=\sum_{i=1}^5\bigl(\lambda r_i(t)+q(t)\bigr)
	=\lambda\sum_{i=1}^5r_i(t)+5q(t)=5q(t).
	\]
	Hence $q(t)=0$ on a dense open set and therefore identically. It remains only to compare the surviving powers of $u$, which gives exactly the vertical images stated in the table. The same comparison of the fourth-power coefficient excludes the translation term for an automorphism acting nontrivially on the base.

	We use the Cox coordinates of \S\ref{subsec: branch equation}, in the order $(x_1,x_2,x_3,x_4)$, and adjoin the double-cover coordinate $w$ with $w^2=x_2f_0$ and $\deg w=(5,3)$. The characteristic quasitorus acts by
	$$
	(x_1,x_2,x_3,x_4,w)\longmapsto
	(\lambda x_1,\mu x_2,\lambda x_3,\lambda^2\mu x_4,\lambda^5\mu^3w),
	\qquad (\lambda,\mu)\in(\mathbb C^*)^2.
	$$
	We choose representatives fixing $x_2$; the remaining equivalence is given by $\mu=1$. A matrix acting on $(x_1,x_3)$ and fixing $x_2,x_4,w$ acts faithfully on the double cover: a scalar matrix gives the identity only if $\lambda^2=\lambda^5=1$, hence $\lambda=1$. In particular, $-I$ acts as the double-cover involution $\tilde{\sigma}$ of Lemma~\ref{lem: K0 deck}.

	We give the explicit equations. In the first case, take
	\[
	R:\quad x_2\bigl(x_2^5(x_1^{10}+x_3^{10})+c x_4^5\bigr)=0,\qquad c\neq 0.
	\]
	The vertical scaling $x_4\mapsto\zeta_5x_4$ gives $\overline{K}_0\cong C_5$, so $|K_0|=10$. The ten roots of $x_1^{10}+x_3^{10}$ have projective stabilizer $D_{20}$. Let $\zeta$ be a primitive tenth root of unity. The matrices
	$$
	A=\begin{pmatrix}\zeta&0\\0&1\end{pmatrix},\qquad
	B=\begin{pmatrix}1&0\\0&\zeta\end{pmatrix},\qquad
	T=\begin{pmatrix}0&1\\1&0\end{pmatrix}
	$$
	preserve the equation while fixing $x_2,x_4,w$. They generate a faithful group $(C_{10}\times C_{10})\rtimes C_2$ of order $200$, with $TAT=B$. Conversely, every automorphism has horizontal image in $D_{20}$ and vertical kernel of order $10$, so this is the full automorphism group.
	In the second case, with $\Phi_8$ the octahedral invariant of degree $8$ above, take
	\[
	R:\quad x_2x_4\bigl(a x_2^4\Phi_8+c x_4^4\bigr)=0,\qquad ac\neq 0.
	\]
	The zero set of $\Phi_8$ is the octahedral $8$-point orbit, whose stabilizer in $\mathrm{PGL}_2(\mathbb{C})$ is the octahedral group $S_4$ (containing, for instance, the order-$4$ rotation $[x_1:x_3]\mapsto[ix_1:x_3]$). We use the standard linearization obtained from the binary octahedral group $BO_{48}\subset\mathrm{SL}_2(\mathbb{C})$: the explicit representation and its invariant ring are given in \cite[\S2.3.1--2.3.2, pp.~6--7]{Freudenburg2025Presentations}, where $\Phi_8=x_1^8+14x_1^4x_3^4+x_3^8\in\mathbb{C}[x_1,x_3]^{BO_{48}}$ is the invariant denoted $v_1$; see also \cite[Theorem~3.3(f), p.~9]{Freudenburg2025Presentations}. The group $BO_{48}$ acts faithfully on the double cover by its standard matrices on $(x_1,x_3)$, fixing $x_2,x_4,w$, and its central element $-I$ acts as $\tilde{\sigma}$. The vertical image is $C_4$, and the automorphism
	$$
	J:(x_4,w)\longmapsto(ix_4,\zeta_8w),\qquad \zeta_8^2=i,
	$$
	has order $8$ and satisfies $J^4=\tilde{\sigma}$. It commutes with $BO_{48}$, and $BO_{48}\cap\langle J\rangle=\{1,\tilde{\sigma}\}$, since an element in the intersection has trivial projective base action. Thus these automorphisms generate
	$$
	(BO_{48}\times C_8)/\langle(-I,J^4)\rangle
	$$
	of order $192$. Every horizontal automorphism lies in the octahedral stabilizer $S_4$, and the vertical kernel has order $8$; hence this is the full automorphism group.
	In the third case, take
	$$
	R:\quad x_2x_4\bigl(x_4^4+x_2^2x_4^2\Phi+x_2^4\Phi^2\bigr)=0.
	$$
	Comparison of coefficients excludes fiber translations and shows that a base matrix $M$ lifts to an automorphism of the branch pair precisely when $\Phi(Mx)=b^2\Phi(x)$, with $x_2$ fixed and $x_4\mapsto bx_4$. Such an automorphism lifts to the double cover as $w\mapsto dw$, where $d^2=b^5$. The Cox equivalence is exactly the one defining $G_\Phi$, so the full automorphism group is $G_\Phi$. For the identity base transformation, take $M=I$; then $b^2=1$, and $(x_4,w)\mapsto(-x_4,iw)$ generates $K_0\cong C_4$. The four distinct roots of $\Phi$ have equianharmonic cross-ratio, so their projective stabilizer is $A_4$. Every element of this stabilizer preserves $\Phi$ up to a scalar, and square roots in $\mathbb C$ provide the required $b$ and $d$. Thus $\overline{K}_0\cong C_2$, $H_0\cong A_4$, and
	$$
	|\mathrm{Aut}(S)|=|G_\Phi|=4\cdot12=48.
	$$

	To identify $G_\Phi$ abstractly, the unique normalization $\lambda=b^2/d$ identifies it with $K_\Phi=\{N\in\mathrm{GL}_2(\mathbb C):\Phi(Nx)=\Phi(x)\}$. Let $BT_{24}=\widetilde{A}_4\subset\mathrm{SL}_2(\mathbb C)$ be the inverse image of the tetrahedral stabilizer. Write $\Phi(hx)=\chi(h)\Phi(x)$ for $h\in BT_{24}$. The character factors through $A_4$ and satisfies $\chi^3=1$, as in Lemma~\ref{lem: A4 semiinvariant}. Hence $\eta(h)=\chi(h)^{-1}h$ preserves $\Phi$. This representation is faithful: a scalar element of $BT_{24}$ is $\pm I$, and $\chi(\pm I)=1$. The scalar subgroup of $K_\Phi$ is $\mu_4 I$. It commutes with $\eta(BT_{24})$, their intersection is exactly $\{I,-I\}$, and together they generate $K_\Phi$ because $\eta(BT_{24})$ surjects onto the projective stabilizer $A_4$. Thus the map $(h,J^k)\mapsto\eta(h)i^kI$ gives
	$$
	G_\Phi\cong(BT_{24}\times C_4)/\langle(-I,J^2)\rangle,
	$$
	where $J$ generates $C_4$.

	In the fourth case ($\overline{K}_0=1$), take
	\[
	R:\quad x_2\bigl(x_4^5+a\,x_1^8x_2^4x_4+b\,x_2^5x_3^{10}\bigr)=0,\qquad ab\neq 0\ \text{general}.
	\]
	The order-$40$ automorphism
	\[
	g\colon\quad [x_1:x_3]\mapsto[x_1:\zeta_{40}x_3],\qquad x_4\mapsto i\,x_4
	\]
	of $\Sigma_2$ (with $x_2$ fixed) multiplies each of the three terms of the branch section by $i$, hence preserves $R$; for general $a,b$ its base image generates $H_0=\langle g\rangle\cong C_{40}$. The surviving $x_4$-powers $0,1,5$ force any vertical scaling to be trivial, so $\overline{K}_0=1$. As $g$ scales the branch section by $i$, it lifts to the double cover $w^2=F$ as $\tilde g\colon w\mapsto\zeta_8 w$ (with $\zeta_8^2=i$); since $\tilde g^{40}=\mathrm{id}$ whereas $\tilde g^{20}$ acts as $[x_1:x_3]\mapsto[x_1:-x_3]$ on the base, $\tilde g$ has order $40$ and $\langle\tilde g\rangle$ does not contain the deck involution $\tilde{\sigma}$ (which is base-trivial). Hence $\langle\tilde g,\tilde{\sigma}\rangle\cong C_{40}\times C_2$ and
	\[
	\mathrm{Aut}(S)\supseteq C_{40}\times C_2,\qquad |\mathrm{Aut}(S)|=80.
	\]

	\medskip
	\noindent\emph{Admissibility of the displayed branch divisors.} We check that each displayed $R=\Delta_0+R_0$ is reduced, satisfies $R_0\cap\Delta_0=\varnothing$, and has only negligible singularities, hence defines a $(1,2)$-surface. In every case the restriction of $R_0$ to $\Delta_0=\{x_2=0\}$ is a nonzero multiple of $x_4^5$, which has no zero on $\Delta_0$ in the Cox quotient; thus $R_0\cap\Delta_0=\varnothing$ and all singularities of $R$ lie away from $\Delta_0$. We first work in the affine chart $x_1x_2\neq 0$ with $t=x_3/x_1$ and $u=x_4/(x_2x_1^2)$, writing $\phi_8,\phi$ for the dehomogenizations of $\Phi_8,\Phi$. For the first three equations, invariance under $x_1\leftrightarrow x_3$ makes the calculation on the chart $x_2x_3\neq0$ identical. The fourth equation is not symmetric, so its second chart is checked separately below.

	For $\overline{K}_0\cong C_5$ the affine equation of $R_0$ is $1+t^{10}+cu^5=0$; the system $1+t^{10}+cu^5=10t^9=5cu^4=0$ forces $t=u=0$, whence $1=0$, a contradiction, so $R_0$ is smooth. For $\overline{K}_0\cong C_4$ it is $u\bigl(a\phi_8+cu^4\bigr)=0$ with $\phi_8=1+14t^4+t^8$; since $\Phi_8$ is square-free the curve $a\phi_8+cu^4=0$ is smooth, and it meets $u=0$ over the simple zeros of $\phi_8$ in ordinary nodes ($us'=0$), so the singularities are of type $A_1$. For $\overline{K}_0\cong C_2$, let $\omega$ be a primitive cube root of unity. The moving branch factors as
	$$
	f_0=x_4(x_4^2-\omega x_2^2\Phi)(x_4^2-\omega^2x_2^2\Phi).
	$$
	The quartic $\Phi$ has four distinct roots, all finite and nonzero in either base chart. Each component is smooth, and the three components meet only at $x_4=0$ over those roots. The moving branch is therefore reduced with exactly four singular points. At each of them take $T=\phi$ and $u=x_4/(x_2x_1^2)$ in the first base chart. The local double-cover equation, after trivializing the covering line bundle, is
	$$
	w^2=u(u^4+Tu^2+T^2)
	   =u\left(T+\frac12u^2\right)^2+\frac34u^5,
	$$
	a rational double point of type $D_6$. After the blow-up $T=uv$, the ordinary strict transform of the branch has equation $u^2+uv+v^2=0$, an ordinary node, whereas the adjusted branch has equation $u(u^2+uv+v^2)=0$, an ordinary triple point with three distinct tangents. On the other blow-up chart $u=Tv$, the adjusted branch is $Tv(1+Tv^2+T^2v^4)=0$ and has an ordinary node; the ordinary strict transform is smooth there. Resolving the adjusted triple point leaves only ordinary nodes. Hence these branch singularities are negligible.

	For $\overline{K}_0=1$, $R_0$ is $u^5+au+bt^{10}=0$; here $g_u=5u^4+a$ and $g_t=10bt^9$ vanish together only if $t=0$ and $u^4=-a/5$, where $g=u(u^4+a)=\tfrac{4a}{5}u\neq0$, so $R_0$ is smooth. On the second chart put $s=x_1/x_3$ and $v=x_4/(x_2x_3^2)$; the fourth equation becomes $v^5+as^8v+b=0$. A singular point would satisfy this equation together with $5v^4+as^8=0$ and $8as^7v=0$. Since $ab\neq0$, the last equation forces $s=0$ or $v=0$, and either choice makes the other derivative force $s=v=0$, contradicting the equation $b=0$. Thus the fourth curve is smooth also over $t=\infty$.

	In each case $R$ is therefore reduced with only negligible singularities and $R_0\cap\Delta_0=\varnothing$, so the converse construction of Section~\ref{subsec: geometry} gives a $(1,2)$-surface $S$. The full groups in the first three rows were determined above by their exact vertical kernels and horizontal stabilizers. In the fourth row, the subgroup of order $80$ attains the upper bound of Theorem~\ref{thm: main}(iv), so it is the full automorphism group. Thus all four bounds of Theorem~\ref{thm: main} are attained.
\end{proof}

\begin{remark}\label{rmk: sharpness}
	Proposition~\ref{prop: large branch examples} realizes the sharp horizontal bounds $20,24,12,40$ by $D_{20}$, $S_4$, $A_4$, and $C_{40}$ in the respective cases $\overline{K}_0\cong C_5,C_4,C_2,1$.
\end{remark}

\begin{remark}\label{rmk: examples replace section}
	Wen's bound is attained by the surface in Proposition~\ref{prop: large branch examples} with $\overline{K}_0\cong C_5$, whose full automorphism group is $(C_{10}\times C_{10})\rtimes C_2$ of order $200$.
\end{remark}

\section*{Acknowledgements}
I would like to thank Songbo Ling for inviting me to visit Shandong University and for the discussions we had during my stay. I am also grateful to Yifei Chen for answering my questions concerning automorphism groups of toric varieties. I am grateful to Akiyoshi Sannai for pointing out two errors in an earlier version of this paper, concerning Theorem~\ref{thm: main}(iii) and the group identifications in Proposition~\ref{prop: large branch examples}, and for communicating the example with $|\mathrm{Aut}(S)|=48$.

\clearpage
\bibliographystyle{amsalpha}
\bibliography{refs}

\medskip
\noindent
School of Mathematics and Statistics, Yunnan University,\\
South Section, East Outer Ring Road, University Town, Chenggong District,\\
Kunming, Yunnan 650500, China

\noindent
\textit{Email address:} \href{mailto:zhaoh@ynu.edu.cn}{zhaoh@ynu.edu.cn}

\end{document}